\documentclass[12pt]{amsart}

\usepackage[all]{xy}
\usepackage[latin1]{inputenc}
\usepackage{hyperref}
\usepackage{amssymb}
\newcommand{\C}{{\mathbb C} }
\newcommand{\cB}{{\mathcal B} }
\newcommand{\cC}{{\mathcal C} }
\newcommand{\cE}{{\mathcal E} }
\newcommand{\cL}{{\mathcal L} }
\newcommand{\cM}{{\mathcal M} }
\newcommand{\cO}{{\mathcal O} }
\newcommand{\cT}{{\mathcal T} }
\newcommand{\cW}{{\mathcal W} }
\newcommand{\cX}{{\mathcal X} }
\newcommand{\cZ}{{\mathcal Z} }
\newcommand{\cH}{{\mathcal H} }
\newcommand{\cK}{{\mathcal K} }
\newcommand{\wt}{\widetilde}
\newcommand{\pt}{\partial}
\DeclareMathOperator{\Der}{Der}
\DeclareMathOperator{\Hilb}{Hilb}
\def\ol#1{{\overline{#1}}}
\newtheorem{theorem}{Theorem}
\newtheorem{lemma}{Lemma}
\newtheorem{proposition}{Proposition}
\newtheorem{definition}{Definition}
\def\ks{Ko\-dai\-ra-Spen\-cer }
\def\ka{K{\"a}h\-ler }
\def\wp{Weil-Pe\-ters\-son }
\def\tei{Teich\-mül\-ler }
\def\ii{\sqrt{-1}}
\def\ddb{\sqrt{-1}\partial\overline{\partial}}
\def\we{\wedge}
\def\cinf{$C^\infty$}

\begin{document}

\title{K\"ahler structure on Hurwitz spaces}

\author[R.\ Axelsson]{Reynir Axelsson}
\address{Raunv\'{i}sindastofnun H{\'a}sk{\'o}lans, Dunhaga 3, is-107
Reykjav{\i}k, {\'I}sland}

\email{reynir@raunvis.hi.is}

\author[I.~Biswas]{Indranil Biswas}

\address{School of Mathematics, Tata Institute of Fundamental
Research, Homi Bhabha Road, Bombay 400005, India}

\email{indranil@math.tifr.res.in}

\author[G.~Schumacher]{Georg Schumacher}

\address{Fachbereich Mathematik und Informatik,
Philipps-Universit\"at Marburg, Lahnberge,
Hans-Meerwein-Strasse, D-35032 Marburg,Germany}

\email{schumac@mathematik.uni-marburg.de}

\date{}



\begin{abstract}
The classical Hurwitz spaces, that parameterize compact Riemann surfaces equipped with covering maps to ${\mathbb P}_1$ of fixed numerical type with simple branch points, are extensively studied in the literature. We apply deformation theory, and present a study of the \ka structure of the Hurwitz spaces, which reflects the variation of the complex structure of the Riemann surface as well as the variation of the meromorphic map. We introduce a generalized \wp \ka form on the Hurwitz space. This form turns out to be the curvature of a Quillen metric on a determinant line bundle. Alternatively, the generalized \wp \ka form can be characterized as the curvature form of the hermitian metric on the Deligne pairing of the relative canonical line bundle and the pull back of the anti-canonical line bundle on ${\mathbb P}_1$. Replacing the projective line by an arbitrary but fixed curve $Y$, we arrive at a generalized Hurwitz space with similar properties. The determinant line bundle extends to a compactification of the (generalized) Hurwitz space as a line bundle, and the Quillen metric yields a singular hermitian metric on the compactification so that a power of the determinant line bundle provides an embedding of the Hurwitz space in a projective
space.
\end{abstract}

\maketitle \tableofcontents

\section{Introduction}

\subsection {Hurwitz spaces}

The concept of a Riemann surface of an algebraic function as a branched covering of the complex projective line ${\mathbb P}_1$ was expounded by Riemann with astonishing clarity, despite predating the development of the abstract notion of a topological space, in a series of four papers in 1857. An $n$-sheeted branched covering  of ${\mathbb P}_1$ is uniquely determined, up to an isomorphism, by its branch points (there are only finitely many of them), and a finite number of certain combinatorial data, which may be described in several different fashions. The simplest case is that of {\it simple} coverings. These are $n$-sheeted branched coverings having at least $n-1$ points over each point of ${\mathbb P}_1$; in other words, all ramification points are simple, and there is at most one ramification point lying over each point of ${\mathbb P}_1$. These coverings were extensively studied by Hurwitz \cite{hu}. He showed that $n$-sheeted simple branched coverings with $b$ branch points can be made into a complex analytic manifold $\cH^{n,b}$, and by a thorough examination of the combinatorial data he could show that $\cH^{n,b}$ is connected. A
modern account of the theory introduced by Hurwitz was given by Fulton in
\cite{f}.

A finite, branched covering of ${\mathbb P}_1$ may be defined as a continuous mapping $\beta : X \to  {\mathbb P}_1$, where $X$ is a compact connected topological space, such that for every point $x \in  X$ there are open neighborhoods $U$ of $x$ in $X$ and $V$ of $\beta(x)$ in ${\mathbb P}_1$, together with homeomorphisms
$$
g : U\to \Delta  ,\quad  h : V \to  \Delta
$$
onto the unit disk $\Delta  \subset  \C$ such that $\beta(U)\subset V$ and $h\circ\beta\circ g^{-1}(z) = z^k$ for some integer $k\geq 1$. This $x$ is a simple ramification point if $k = 2$. Given a branched covering $\beta$, there is a unique complex structure on $X$ that makes the mapping $\beta$ holomorphic. An element of the Hurwitz space $\cH^{n,b}$ is a representative of an equivalence class of simple $n$-sheeted branched coverings of ${\mathbb P}_1$. Two such coverings
$$
\beta_1 : X_1 \to  {\mathbb P}_1~ ~\text{ and }~  ~ \beta_2 : X_2 \to  {\mathbb P}_1
$$
are called \textit{equivalent} if there is a homeomorphism $\chi : X_1 \to X_2$ such that $\beta_1 = \beta_2\circ\chi$; such a mapping $\chi$ is necessarily biholomorphic with respect to the above mentioned complex structures on $X_1$ and $X_2$. Given a simple $n$-sheeted covering $\beta : X \to  {\mathbb P}_1$ with branch points $p_1,\cdots, p_b$, choose pairwise disjoint open disks $D_1,\cdots ,D_b$ in ${\mathbb P}_1$ around the branch points, and choose for each point
$$
q := (q_1,\cdots,q_b) \in D_1\times\cdots\times D_b
$$
a diffeomorphism $\alpha : {\mathbb P}_1 \to {\mathbb P}_1$ carrying $p_j$ to $q_j$ such that $\alpha$ agrees with the identity map of ${\mathbb P}_1$ on the complement of the disks. Set $\beta_q := \alpha\circ\beta$. It is easily seen that $\beta_q$ is a simple $n$-sheeted covering, with branch points $q_1  ,\cdots  ,q_b$, and that the equivalence class of $\beta_q$ depends only on $q$, in particular, the equivalence class is independent of the choice of the diffeomorphism $\alpha$. The subsets
$$
U(\beta;  D_1,\cdots,D_b) :=  \{\beta_q \mid  q\in D_1\times \cdots \times D_b\}  \subset  \cH^{n,b}
$$
form a basis for a topology on $\cH^{n,b}$ such that the mapping
$$
U(\beta; D_1,\cdots,D_b) \to  D_1\times \cdots \times D_b  , \quad \beta_q \mapsto  q
$$
is a homeomorphism. Letting $D$ to be the image of $D_1\times\cdots\times D_b$ in the $b$-fold symmetric product $\text{Sym}^b( {\mathbb P}_1)$, we see that $D$ is evenly covered by the mapping $\delta : \cH^{n,b} \to  U$, where $U \subset \text{Sym}^b({\mathbb P}_1)$ is the locus of $b$ distinct points, and $\delta$ maps a simple branched covering $\beta$ to its set of branch points; the set $U$ may be identified with a Zariski open subset of the complex projective space ${\mathbb P}_b = \text{Sym}^b({\mathbb P}_1)$. Thus $\delta : \cH^{n,b} \to  U$ is a surjective finite sheeted unbranched covering map, and hence there is a unique complex structure on $\cH^{n,b}$ that makes the mapping $\delta$ holomorphic.

A nontrivial automorphism of a simple branched covering of ${\mathbb P}_1$ must keep the ramification points fixed and hence must be an involution, assuming that $b > 0$ (because the ramification number is two, any automorphism of the covering is an involution in a neighborhood of a ramification point in the domain Riemann surface). After taking the quotient by the involution we obtain an unbranched covering of ${\mathbb P}_1$, which must be a biholomorphic map. Thus such an automorphism can only exist (and exists) for $n= 2$. It follows that for $n\geq 3$, the simple branched coverings $\beta_s : X_s \to {\mathbb P}_1$, $s \in  \cH^{n,b}$, fit naturally together to form a family $\beta : \cX \to {\mathbb P}_1\times \cH^{n,b}$ of mappings over $\cH^{n,b}$, where $\cX$ is the union of the surfaces $X_s$ provided with a complex structure such that $\beta$ is holomorphic. This means that $\beta$ is a {\it deformation} of
$X$ {\it relative} to $\mathbb P_1$.

We can define a category of deformations of simple branched coverings of ${\mathbb P}_1$ parameterized by a complex space. An object of this category is a mapping $\beta_S : \cX \to   {\mathbb P}_1\times S$ of spaces over an arbitrary complex space $S$ (the parameter space) such that $\cX$ is flat over $S$, and for every $s$ in $S$, the induced mapping $\beta_s : \cX_s \to {\mathbb P}_1\times\{s\}$ of fibers is a simple branched covering. A morphism in this category is a pullback diagram induced by such a mapping via a holomorphic mapping $T \to   S$ of the parameter spaces. One can then show that the Hurwitz family represents a final object in this category for $n \geq  3$. This was proved by Fulton in a general algebraic setting
\cite{f}.

Our objective here is to consider Hurwitz spaces from the point of view of deformation theory.

\subsection{Description of results}

Generalizations of Hurwitz spaces have been carried out by various authors. The simplest generalization is to consider branched coverings $X \to Y$ of an arbitrary but fixed compact Riemann surface $Y$. A construction of these moduli spaces is described in Section~\ref{se:genHu}.

We compute the tangent spaces of these generalized Hurwitz spaces in Section~\ref{se:tangcoh}, and relate it to the space of infinitesimal deformations of $X$ as well as to the infinitesimal deformations of the simple branched covering $X \to Y$ with $Y$ fixed. Equip $X$ with the hyperbolic metric and also equip the fixed Riemann surface $Y$ with a metric of constant curvature. Given these metrics, in Section~\ref{se:WP} we introduce a generalized \wp metric on the base of a local, universal family in a functorial way so that it descends to the induced moduli stack. It descends  in the orbifold sense to the corresponding generalized Hurwitz space, which exists
as a coarse moduli space.

We show that the generalized \wp form satisfies a fiber integral formula, in particular, it is a \ka form according to Theorem~\ref{th:fibint}. We prove that the generalized \wp form is the curvature of a certain holomorphic hermitian line bundle. This holomorphic hermitian line bundle $\lambda$ can be realized as a determinant line bundle equipped with a Quillen metric (Theorem~\ref{th:linebundle}). This generalized \wp form also can be realized as the Deligne pairing $\langle\cK_{\cX/\cH}, \beta^* {\mathcal L}\rangle$ equipped with the natural hermitian metric, where ${\mathcal L}$ is a certain holomorphic hermitian (orbifold) line bundle on $Y$ (Section~\ref{se:delpa}). Namely $\cL$ is $K_Y$ or $K^{-1}_Y$, if ${\rm genus}(Y) >  1$ or ${\rm genus}(Y)  =   0$ respectively, and it defines a principal polarization in the case of ${\rm genus}(Y)  =  1$. The generalized \wp form extends, as a closed positive current, to a compactification of the Hurwitz space. From this fact it can be deduced using analytic methods that a power of $\lambda$ embeds the generalized (open) Hurwitz space into a projective space (Theorem~\ref{th:embed}).

{\em Acknowledgement.} The authors would like to thank the referee for helpful remarks.

\section{Relative tangent cohomology for Hurwitz spaces}\label{se:tangcoh}

We recall some facts about tangent cohomology that was introduced by Palamodov in \cite{pal}, and apply these to the Hurwitz spaces.

Let $A \to  B$ be a morphism of analytic $\C$-algebras, and let $M$ be a finite $B$-module. Then we denote by $T^q(B/A,  M)$ and $T^q(B,  M)$ the relative tangent functor and the absolute tangent functor respectively; we recall that $T^0(B/A, M) = \Der_A(B,  M)$ is the module of $A$-linear $M$-valued derivations of $B$, while $T^1(B/A,  M)$ stands for the $B$-module of equivalence classes of analytic extensions of the $\C$-algebra $B$ by $M$ over $A$. As usual, the obstructions lie in $T^2(B/A, M)$. According to \cite[Proposition 1.6]{pal}, there is a long exact sequence of $B$-modules
\begin{gather}\label{eq:long}
0 \to   T^0(B/A, M) \to  T^0(B, M)  \to  T^0(A, M)  \to  \\ \nonumber
\strut \hspace{1cm} T^1(B/A, M)  \to   T^1(B, M)  \to   T^1(A,
M)  \to  \cdots  ,
\end{gather}
where $M$ is considered as an $A$-module using the homomorphism $A  \to  B$. For a holomorphic map $\beta_0 :  X \to  Y$ together with a coherent $\cO_X$-module $\cM$, these objects give rise to coherent sheaves, which are denoted by $\cT^q_{X/Y}(\cM)$ etc. If $\cM = \cO_X$, then the entry $\cM$
is dropped from the notation.

Let $\beta_0 : X \to Y$ be a non-constant holomorphic map between compact Riemann surfaces. We will denote the genus of $X$ and $Y$ by $p(X)$ and $p(Y)$ respectively. We assume $p(X)>1$. In the classical case of Hurwitz spaces, we have $Y = {\mathbb P}_1$. Assume that $\beta_0$ has only simple ramification points (mapping to different branch points) so that for the ramification divisor $B = \sum_1^b z_j$ on $X$ the points $\beta_0(z_j)$
are distinct.

Let $D = (\{0\}, \C[Z]/(Z^2))$ be the double point. Infinitesimal deformations $\psi$ of $X/Y$ are isomorphism classes of Cartesian diagrams
\begin{equation}\label{eq:def1}
\xymatrix{X \ar[r] \ar[d]^{(\beta_0,0)}  & \cX \ar[d]^{(\beta,f)}   \\
 Y \times \{0\} \ar@{^(->}[r] &Y \times D}
\end{equation}
where $f$ is a flat proper morphism. Let
\begin{equation}\label{m1}
\nu(\psi)  \in  H^1(X,  T_X)
\end{equation}
be the induced deformation of $X$, i.e.\ $\nu(\psi)$ is given by
$$
\xymatrix{X \ar[r] \ar[d]  & \cX \ar[d]^f   \\  \{0\} \ar@{^(->}[r] & D}
$$
and $\nu(\psi) =  0$ if the diagram in \eqref{eq:def1} is isomorphic to
\begin{equation}\label{eq:def3}
\xymatrix{X \ar[r] \ar[d]& X\times D \ar[d]\\  Y \times \{0\}
\ar@{^(->}[r]
&Y \times D}
\end{equation}

It is straightforward to deduce the following lemma.

\begin{lemma}\label{le:tanshe}
The long exact sequence of tangent sheaves induced by \eqref{eq:long} gives rise to a short exact sequence
$$
0  \to  \cT^0_{X}  \to  \beta_0^*\cT^0_Y  \to  \cT^1_{X/Y}  \to  0
.
$$
The coherent sheaf $\cT^1_{X/Y}$ is supported on the set of ramification points $\{z_j\}_{j=1}^b$, and it is isomorphic to
$\oplus^b_{j=1} \cO_X/\mathfrak m_{z_j}$. Furthermore,
$$
\beta_0^*\cT^0_Y  \simeq  \cT^0_X(B)  .
$$
\end{lemma}

Tangent cohomology groups and the cohomology with values in the tangent cohomology sheaves are related by a spectral sequence $E^{p,q}_2 = H^p(X, \cT^q_{X/Y})$ that converges to $T^{p+q}(X/Y)$, (cf.\ \cite[Theorem 4.1]{pal}). We first note $T^0(X/Y)=H^0(X,\cT^0_{X/Y})$. We use the standard sequence involving terms of low order that is induced by a spectral sequence of the above type. Because of $\dim X=1$ the latter yields the
following exact sequences.
$$
0\to H^1(X,\cT^0_{X/Y})\to T^1(X/Y) \to H^0(X,\cT^1_{X/Y})\to 0.
$$
Furthermore the standard sequence (in dimension one) implies
$$
T^2(X/Y)= H^1(X, \cT^1_{X/Y}).
$$
Using Lemma~\ref{le:tanshe} and the above argument we obtain the following:
\begin{lemma}\label{le:tanco}
The following isomorphisms hold:
$$
T^0(X/Y) =0 , \quad T^1(X/Y)=H^0(X, \cT^1_{X/Y}), \quad T^2(X/Y)=0.
$$
\end{lemma}
In particular, $T^1(X/Y) \simeq  \C^b$, {\em all infinitesimal deformations of $X/Y$ are unobstructed}, and {\em semi-universal deformations are universal}. Now Lemma~\ref{le:tanshe} implies the following:

The tangent space $T^1(X/Y)=H^0(X,\cT^1_{X/Y})$ of the Hurwitz stack given at a point $\beta_0:X\to Y$ is related to the space of infinitesimal deformations of $X$ on one hand and to the space of infinitesimal deformations of $\beta_0$ (with $X$ and $Y$ fixed) on the other hand in the following way.

\begin{proposition}\label{le1}
There exists a long exact sequence
$$
0\to H^0(X, \beta_0^*\cT^0_{Y}) \to H^0(X,\cT^1_{X/Y})\stackrel{\nu}{\to}
H^1(X,\cT^0_X)\to H^1(X,\beta_0^*\cT^0_Y)\to 0 ,
$$
where $\nu$ was constructed in \eqref{m1}.
\end{proposition}


Note that $\textrm{deg}((\beta_0^*\cT_Y)^*\otimes K_X) = 4p(X)-4 -b$. Hence $H^1(X,  \beta_0^*\cT^0_Y) =  0$ by Serre duality if $b > 4p(X)-4$. So we have the following statement.

\begin{lemma}\label{le:largeorder}
The obstructions against extending an infinitesimal deformation of $X$ to a deformation of $X$ over $Y$ are the elements of $H^1(X, \beta_0^*\cT^0_Y)$. This space of obstructions vanishes if $b > 4p(X)-4$. In this case the space $H^0(X, \beta_0^*\cT^0_{Y})$ is zero, unless $Y=\mathbb P_1$.
\end{lemma}

At the end of this section, we will now give another approach to the tangent cohomology $T^\bullet(X/Y)$, which will not be used later. We consider the more general situation from \cite{RW}. We denote the above homomorphism $\cT^0_X  \to \beta^*_0\cT^0_Y$ by $d\beta_0$. This map can be turned into a complex
$$
\xymatrix{
\cC^\bullet : & 0  \ar[r]& \cC^0 \ar[r]^{d\beta_0} &   \cC^1 \ar[r] &  0 }
$$
which fits into a short exact sequence $0 \to  {\cC'}^\bullet \to \cC^\bullet \to  {\cC''}^\bullet \to  0$ of complexes:
$$
\xymatrix{
 & 0 \ar[d]& 0 \ar[d] & \\
{\hspace{-10mm}{\cC'}^\bullet : \hspace{5mm}}0 \ar[r] & \cT^0_X
\ar[r]^{d\beta_0}\ar[d] & d\beta_0(\cT^0_X) \ar[r]\ar[d] & 0\\
{\hspace{-10mm}\cC^\bullet : \hspace{5mm}} 0 \ar[r] &\cT^0_X
\ar[r]^{d\beta_0}\ar[d] & \beta^*_0 \cT^0_Y \ar[r]\ar[d] & 0 \\
{\hspace{-10mm}{\cC''}^\bullet : \hspace{5mm}}0\ar [r] & 0 \ar[r]\ar[d] &
\beta^*_0\cT^0_Y/ d\beta_0(\cT^0_X) \ar[r]\ar[d] & 0\\
& 0 & 0 & }
$$
Since $d\beta_0$ is injective, the cohomology of the complex ${\cC'}^\bullet$ vanishes. So for all $q$,
$$
{\mathbb
H}^q(\cC^\bullet) = {\mathbb
H}^q({\cC''}^\bullet)  =  H^{q-1}(X,  \beta^*_0\cT^0_Y/
d\beta_0(\cT^0_X) )  ,
$$
where $\mathbb H$ denotes the hypercohomology. Moreover, the following holds:

\begin{proposition}
The hypercohomology ${\mathbb H}^\bullet(\cC^\bullet)$ can be
identified with the tangent cohomology $T^\bullet(X/Y)$. In
particular, the tangent space to the Hurwitz space at $\beta_0$
is identified with ${\mathbb H}^1(\cC^\bullet)$.
\end{proposition}

\section{\wp structure}\label{se:WP}

\subsection{The pairing}\label{ss:pairing}
As before, take $\beta_0 :  X \to  Y$, with $p(X)>1$. We consider smooth base spaces $0 \in S \subset \C^r$ and deformations over $S$ given by maps $(\beta, f) : \cX \to Y\times S$ together with an isomorphism from $X$ to the distinguished fiber $f^{-1}(0)$ that takes $\beta_0$ to $\beta\vert_{f^{-1}(0)}$. For any $s\in S$, the fiber $f^{-1}(s)$ will be denoted by $\cX_s$. Let $z$ be a local holomorphic coordinate on the fibers, and $s = (s^1,\cdots,s^r) \in  \C^r$. We now apply the construction
of a natural hermitian metric on $\cK_{\cX/S}$ from \cite[Section~4]{sch}.

We equip the fibers $\cX_s$ with the unique hyperbolic metrics
$$
\omega_{\cX_s}= g(z,s) \ii dz \wedge d\ol{z}
$$
of constant Ricci curvature $-1$, which are of class \cinf\ and depend in a \cinf\ way upon the parameters.

The curvature condition reads
\begin{equation}\label{eq:hypfib}
({\pt^2}\!/{\pt z \ol{\pt z}}) \log g(z,s)  =  g(z,s)  .
\end{equation}
(Since we are applying the notions of \ka geometry, we chose this normalization.) The \ka forms on the fibers of $f$ induce a hermitian metric $g^{-1}(z,s)$ on the relative canonical line bundle $\cK_{\cX/S}$. We denote its negative (real) curvature
form by
$$
\omega_\cX  =   \ddb\log g(z,s)  .
$$
It follows from \eqref{eq:hypfib} that
$$
\omega_\cX|\cX_s =  \omega_{\cX_s}
$$
for all $s \in  S$. We have the following special case of \cite[Theorem~1]{sch}:

\begin{theorem}\label{th:pos}
Let $f : \cX \to  S$ be an effectively parameterized family of curves of genus larger than one. Then $\omega_\cX$ is a \ka
form on $\cX$.
\end{theorem}

Next, we fix a metric $\omega_Y =  h(w)\cdot\ii dw \wedge \ol{dw}$ on $Y$ of constant Ricci curvature $\epsilon = 0$ or $\epsilon = \pm 1$ depending on the genus of $Y$. If $p(Y)  = 1$, we normalize the flat metric $\omega_Y$ by imposing the condition that the total volume is $1$. In the latter case $\omega_Y$ is the curvature form of a positive bundle $(\cL,k)$ on $Y$ that represents the principal polarization. Else $\omega_Y$ is the curvature form of a unique hermitian metric on $\cK_Y^{-1}$ or $ \cK_Y$
depending on whether $p(Y)$ is zero or greater than one.

The metric $h$ on $Y$ induces the hermitian metric $\beta_0^*h$ on the line bundle $\beta_0^*\cT^0_Y$. The symbol $s$ as an index of a tensor stands for one of the coordinate functions $s^j$ on $S$; more generally, $\pt/\pt s = \pt_s$ will stand for an arbitrary (non-zero) tangent vector on $S$.

\medskip
\noindent {\bf Notation.}   Let $(\beta,f) : \cX  \to Y \times S$ be a deformation of $X$ over $Y$ over $S$ in the sense above. In terms of the local coordinate system $(z,s^1,\cdots,s^r)$ on $\mathcal X$, we denote the partial derivatives of $g$ by $g_{z\ol z}$, $g_{i,\ol z}$, $g_{i\ol\jmath}$ etc.; if we restrict ourselves to a $1$-dimensional base space, we denote the components simply by $g_{s,\ol z}$, $g_{s\ol s}$ etc. We denote by $w$ local coordinates, on $Y$, such that $f(z,s)=s$ and $\beta(z,s)=w$,
with $\beta_0(z)=\beta(z,0)$ on $X$.

The derivative of $\beta$ that maps the tangent space of $\cX$ at a point $(z,s)$ to the tangent space of $Y$ at $w=\beta(z,s)$ induces differential forms with values in the holomorphic vector fields  along $\beta$.
\begin{eqnarray} \label{eq:zeta}
  d \beta_{rel}&:=& \zeta:= \zeta^w_z \left.\frac{\pt}{\pt w}\right|_{\beta(z,s)}\!\!\! dz \; \text{ with } \quad
   \zeta^w_z  =  \frac{\pt\beta(z,s)}{\pt z}\; ,
   \\ \label{eq:xi}
   d \beta&=&\left( \zeta^w_z dz+ \xi^w_s ds\right) \left.\frac{\pt}{\pt w}\right|_{\beta(z,s)}\; \text{ with } \quad
   \xi^w_s=  \frac{\pt\beta(z,s)}{\pt s} \; .
\end{eqnarray}
\medskip

We have the differential
$$
\zeta=d\beta_{rel} \in H^0(\cX, (\cT^0_{\cX/S})^{-1}\! \otimes \beta^*\cT^0_Y)= H^0(\cX, \Omega_{\cX/S}(\beta^*(\cT^0_Y))  .
$$
Let $\cB\subset \cX$ be the ramification divisor of $\beta$. The above section $\zeta$ identifies $(\cT^0_{\cX/S})^{-1}\otimes \beta^*\cT^0_Y$ with $\cO_\cX(\cB)$; this identification takes $\zeta$ to the constant function $1 \in  H^0(X, \cO_\cX(\cB))$.

We will briefly recall a description of the {\em classical \wp form}, on the Teich\-m\"ul\-ler/mo\-du\-li space of Riemann surfaces of genus larger that one, in terms of canonical (\cite{siu:canlift}) and horizontal lifts (cf.\ \cite{sch,a-s}). Let
$$
\rho_s  :  T_sS  \to  H^1(\cX_s,
\cT^0_{\cX_s})
$$
be the \ks map associated to the deformation $\nu(\psi)$ (see \eqref{m1}). The harmonic representatives of $\rho_s (\pt_s)$ with respect to the hyperbolic metric on $\cX_s$ are harmonic Beltrami differentials, which we
denote by $\mu_s = \mu^z_{s\ol z}\pt_z \ol{dz}$.

The harmonic Beltrami differentials are closely related to the variation of the metric tensor in a family. We denote by
\begin{eqnarray*}
  v_s= \pt_s + a^z \pt_z \quad \text{where} \quad  a^z_s= -g^{\ol z z}g_{s\ol z}
\end{eqnarray*}
the {\em horizontal} lift of a tangent vector $\pt_s$ of $S$ that is orthogonal to the fiber with respect to $\omega_\cX$ (cf.\ \cite[p.\ 247, (6), (7)]{a-s}). Now the harmonic Beltrami
differential is
$$
\mu_s =  (\ol\pt v)\vert\cX_s  =  \pt_{\ol z}(a^z_s)\pt_z \ol{dz}  .
$$
The canonical $L^2$ inner product of harmonic Beltrami differentials with respect to the hyperbolic metric on the fibers is known as the \wp inner product.

This construction requires a modification for the Hurwitz spaces.

To construct a generalized \wp metric on a Hurwitz space, we again use an approach from higher dimensional theory, namely canonical lifts of tangent vectors introduced by Siu \cite{siu:canlift} that were used in \cite{f-s:extremal} in the form of horizontal lifts. We follow
\cite[Proposition~3]{sch}.

Again, once we are dealing with a hermitian inner product on the tangent space of the base, which is functorial, i.e.\ compatible with base change, we only need to compute the norms of tangent vectors. This means that it is sufficient to look at base spaces of embedding dimension one. Since all deformations are unobstructed,  we may assume without loss of generality that $S$ is a disk in $\C$. In this way
the notation can be simplified.

The tangent space of the Hurwitz space will be treated in terms of the exact sequence from
Proposition~\ref{le1}.

The first step is to construct a semi-definite inner product $G^0_{WP}$ on the image of $T^1(X/Y)=H^0(X, \cT^1_{X/Y})$ in $H^1(X, \cT^0_X)$ that is {\em not} the usual \wp inner product on the image $\nu(\rho_s(\pt_s))$ in $H^1(X,\cT^0_X)$ given in Proposition~\ref{le1}.

Let
\begin{equation}\label{zy}
\varphi := \varphi_{s\ol s} := \langle v_s,v_s\rangle_{\omega_\cX} =
g_{s\ol s} - g_{s\ol z}g_{z \ol s}g^{\ol z z}
\end{equation}
be the pointwise norm of the horizontal lift of $v_s$. Then
$$
\omega^2_\cX  = \omega_\cX \wedge  (\varphi   \ii ds \wedge \ol{ds})  .
$$
The right-hand side depends only on the induced relative \ka form $\omega_{\cX/S}$. Let
$$
\Box_{\omega_s}  =  \Box_{\omega_{\cX_s}}
$$
be the fiberwise Laplacian on differentiable functions for the relative \ka form $\omega_{\cX/S}$. Then the pointwise norm of the induced harmonic Beltrami differential satisfies the
equation
\begin{equation}\label{eq:ell}
\Box_{\omega_s} \varphi + \varphi  =   \|\mu_s\|^2(z,s)  .
\end{equation}

The operator $\Box_{\omega_s} + 1$ is strictly positive.  If $\mu_s \neq  0$, then the function $\varphi$ is everywhere positive by \cite[Proposition 3]{sch}. We define
\begin{equation}\label{ewp}
G^{WP}_{0,s\ol s}(s) :=  \int_{\cX_s}
(\Box_{\omega_s}+1)^{-1}(\|\mu_s\|^2(z,s))   \beta_s^*\omega_Y  ,
\end{equation}
where $\beta_s = \beta\vert\cX_s$. This construction is functorial, since it depends upon Beltrami differentials that are harmonic with respect to the hyperbolic metric on the fibers and the fixed metric on
the base space $Y$. So \eqref{ewp} defines a real $(1,1)$-form on the Hurwitz space.

Because of \eqref{eq:ell} we have
\begin{equation}\label{eq:wp1}
G^{WP}_{0,s\ol s}(s) =  \int_{\cX_s} \varphi(z,s)   \zeta(z,s) \ol \zeta(z,s)
h(\beta(z,s)) \ii dz \we\ol{dz}  .
\end{equation}

In the second step we construct an inner product $G^{WP}_1$ regarding the kernel of $\nu$ in Proposition~\ref{le1}. It has to take into account the case where the fiber $\cX_s$ is fixed. A {\em holomorphic lift} of a  tangent vector $\pt_s$ with values in
$\beta^*\cT^0_Y$ is given by
$$
u_s  =  \pt_s - \xi^w_s(z,s) \pt_w|_{\beta(z,s)}  .
$$
The horizontal lift $v_s$ induces a lift with values in
$\beta^*\cT^0_Y$:
$$
\wt v_s =  \pt_s + a^z_s(z,s)  \zeta^w_z(z,s) \pt_w|_{\beta(z,s)}  .
$$
At this point, we use an approach similar to \cite[(18)]{a-s1} and consider the difference of two lifts of tangent vectors of the base. The difference
$$
U_s  :=  \wt v_s - u_s
$$
is a vector field along the fibers of the map $\beta  :  \cX \to  Y$. The vector field $U_s$ can also be interpreted as the image
\begin{equation}\label{eq:us}
U_s = \beta_{*}(v_s)|\cX_s.
\end{equation}
Namely, using \eqref{eq:xi} and
\eqref{eq:zeta} we get
\begin{equation}\label{eq:us1}
U_s= (a^z_s \zeta + \xi_s) \pt_w|_{\beta(z,s)}.
\end{equation}
Its pointwise norm $\|U_s\|^2(z,s)$ is given in terms of the metric $h$ on $Y$:
\begin{equation}\label{eq:wp2}
\|U_s\|^2(z,s) =   (a^z_s\zeta^w_z + \xi^w_s) (a^\ol z_\ol s \zeta^\ol w_\ol s
+ \xi^\ol w_\ol z)h(\beta(z,s))  .
\end{equation}
We integrate the function $\|U_s\|^2(z,s)$ over the fiber $\cX_s$ with respect to the area element $g\, dA$
of the hyperbolic metric on $\cX_s$. Define
\begin{equation}\label{ewp2}
G^{WP}_{1,s\ol s}(s)   =   \int_{\cX_s} \| U_s\|^2(z,s) \,  g \,  dA  .
\end{equation}
Again the result is intrinsically defined and compatible with base change. It is easy to see that the norms $G^{WP}_1$ and $G^{WP}_0$ give rise to a hermitian product.

\begin{definition}\label{de:pwpr}
The \wp inner product on the tangent space $T^1(\cX_s/Y)=H^0(\cX_s,\cT^1_{\cX_s/Y})$ of the Hurwitz space is defined by its norm
\begin{equation}\label{zy2}
\|\pt_s \|^2_{WP}= G^{WP}_{s\ol s}(s) :=  G^{WP}_{1,s\ol s}(s)+G^{WP}_{0,s\ol s}(s)  ,
\end{equation}
where $G^{WP}_{0,s\ol s}(s)$ and $G^{WP}_{1,s\ol s}(s)$ are defined in \eqref{ewp} and \eqref{ewp2} respectively.
\end{definition}

\subsection{Positivity of the pairing}

\begin{proposition}
The \wp inner product is positive definite if the family
$$
(\beta,f)  :  \cX  \to  Y\times S
$$
is effectively parameterized.
\end{proposition}

\begin{proof}
We follow the notation of this section. We prove the property at a point $s_0  \in  S$. Assume that for a tangent vector $\pt_s \in  T_{S,s_0}$, the value $G^{WP}_{s\ol s}(s_0)$ in \eqref{zy2} vanishes. So both
$G^{WP}_{0,s\ol s}$ and $G^{WP}_{1,s\ol s}$ are zero at the distinguished point.

We consider the former term. Since the operator $(\Box_{\omega_s}+1)^{-1}$ is strictly positive, it follows from \eqref{ewp} that the harmonic Beltrami differential $\mu_s$ is equal to zero (as well as the function
$\varphi_{s\ol s}$). Therefore, the horizontal lift
$$
v_s = \pt_s + a^z_s \pt_z
$$
is holomorphic on the first infinitesimal neighborhood of $\cX_{s_0}$. As the construction is independent of the choice of coordinates, we can pick local holomorphic coordinates on the total space $\cX$ such
that $v_s= \pt_s$ at $s_0$. Now $a^z_s=0$ and $\varphi_{s\ol s}=0$ on the given fiber.

Next, we consider $G^{WP}_{1,s\ol s}(s_0)  = 0$. Since $a^z_s=0$, we have
$$
0 = \|U_s \|^2(z,s_0)  =  \xi^w_s \xi^{\ol w}_{\ol s} h(\beta(z,s_0)),
$$
implying that the infinitesimal deformation of $\beta_0 :  \cX_{s_0} \to  Y$ satisfies the equation
$$
\xi^w_s(z,s_0)\pt_w|_{\beta(z,s_0)} =  0  .
$$
In the chosen coordinates $u_s=\pt_s$ holds at $s_0$, which means that also the induced deformation of $\beta_0$ is infinitesimally trivial.
\end{proof}

\section{\ka property of the \wp structure}

The \wp inner product on the base of a family of compact Riemann surfaces together with simple branched covering maps to a fixed Riemann surface defines a \wp form $\omega^{WP}$ on the base as follows: If $s^1\! ,\cdots,s^r$ are local holomorphic
coordinates on the base, then
$$
\omega^{WP} =  \ii\cdot G^{WP}_{i \ol \jmath} ds^i \wedge ds^{\ol \jmath}  ,
$$
where $G^{WP}_{i \ol \jmath}(s) = G^{WP}(\pt/\pt s_i|_s, \pt/\pt s_j|_s)$.

\begin{theorem}\label{th:fibint}
Let $(\beta,f) : \cX \to  Y \times S$ be a holomorphic family. Let
$$
\omega_\cX = \ii   \partial\overline{\partial} \log g(x,s)
$$
be the curvature form of $\cK_{\cX/S}$ equipped with the metric that is induced by the hyperbolic metric on the fibers of $f$. Let $\omega_Y$ be the earlier defined metric on $Y$ of constant Ricci curvature equal to $\epsilon =0  \text{ or }\pm 1$. Then
the \wp form satisfies
\begin{equation}\label{eq:fibint}
\omega^{WP} =  \int_{\cX/S} \omega_\cX \wedge \beta^*\omega_Y  .
\end{equation}
In particular, the \wp form is K\"ah\-ler.
\end{theorem}

\begin{proof}
The expression on the right hand side of \eqref{eq:fibint} is clearly functorial so that it is enough to prove the equation for one dimensional base spaces $S = \{s\} \subset \C$. Again, we use the letter $s$ also as a subscript to identify the corresponding index of a tensor.

We compute the integrand in \eqref{eq:fibint}:
\begin{small}
\begin{gather*}
\omega_\cX \wedge \beta^*\omega_Y = \ii \left(g_{s\ol s} ds \we \ol{ds} + g_{s\ol z} ds \we \ol{dz} + g_{z\ol
s} dz \we \ol{ds} + g_{z \ol z}dz \we \ol{dz}\right)\\
\strut \hspace{5cm}  \we\ii (h\circ \beta)\left(\zeta dz + \xi_s ds\right)\we\left(\ol \zeta \ol{dz} + \ol
\xi_{\ol s} \ol{ds}
\right)\\
 = \left( g_{s\ol s} \zeta\ol
\zeta - g_{s\ol z} \zeta \ol \xi_s - g_{z\ol s} \xi_s \ol \zeta
+ g_{z\ol z}\xi_s \ol \xi_{\ol s} \right)
(h\circ\beta) \cdot (\ii)^2 dz \we \ol{dz}\we ds \we \ol{ds}\\
= \left(\varphi \zeta\ol\zeta +
(a^z_s\zeta+\xi_s)(a^{\ol z}_{\ol s}\ol \zeta+ \ol\xi_{\ol
s})g_{z\ol z} \right)(h\circ\beta)
\cdot (\ii)^2 dz \we \ol{dz}\we ds \we \ol{ds}
\end{gather*}
\end{small}
We use \eqref{eq:wp1} and \eqref{eq:wp2} which yield \eqref{eq:fibint}.
\end{proof}

\section{Generalized Hurwitz spaces}\label{se:genHu}
In this section we describe an approach to generalized Hurwitz spaces that is related to our construction of the generalized \wp form. For the existence of generalized Hurwitz stacks refer to the article of Harris, Graber, and Starr \cite{hgs}. We consider simple branched coverings $\beta :  X \to Y$, where the Riemann surface $Y$ and the integers $p  = p(X) >  1$ and $\text{degree}(\beta)$ are fixed. The number of branch points will be denoted by $b$ (note that it is fixed). Two such objects $\beta : X  \to  Y$ and $\beta': X'\to Y$ are \textit{isomorphic}, if there exists an isomorphism $\varphi: X \to   X'$ such that $\beta'\circ\varphi = \beta$. Recall from Lemma~\ref{le:tanco} that deformations of Riemann surfaces $X$ with $p(X) > 1$ together with a simple branched covering maps to a fixed Riemann surface $Y$ are unobstructed and that the automorphism groups of such objects
are contained in the automorphism groups of $X$ and hence they are finite.

Let $\cZ_p  \to  \cT_p$ be the universal family over the \tei space of compact Riemann surfaces of genus $p>1$. We introduce a level-$m$ structure, e.g.\ $m = 3$, (cf.\ \cite{Lo,PdJ} for details). We take the subgroup $\Gamma_p[3] \subset \Gamma_p$ of the \tei modular group that acts trivially on $H^1(X,   {\mathbb Z}/3{\mathbb Z})$ (in other words, fixes the level structure). The group $\Gamma_p[3]$ acts on the universal family in an equivariant, proper, discontinuous and free way yielding a projective
universal family
$$
\cZ  =   \cZ^{[3]}_p  \to   \cM^{[3]}_p  =  \wt\cT  .
$$
Let $Y$ be a fixed Riemann surface different from ${\mathbb P}_1$. As before, the fiber over $s \in  \wt\cT$ will be denoted by $\cZ_s$. Let $\cZ \times Y \to \wt\cT$ be the family of Cartesian products over $\wt\cT$. We equip the family $\cZ \to  \wt \cT$ with the (relative) canonical polarization; the Riemann surface $Y$ also carries a holomorphic hermitian line bundle whose curvature is the given K\"ahler
form $\omega_Y$ (up to a numerical factor) inducing a polarization of the product.

Given a simple branched covering $\beta :X \to Y$ with given number of branch points and sheets, where $X$ stands for one fiber $\cZ_s$, the graph of the covering is provided with an induced polarization and Hilbert polynomial. Let $\Hilb  = \Hilb_{\wt\cT}(\cZ\times Y)$ be the Hilbert scheme parameterizing closed subschemes with these data.  Let $\cW \hookrightarrow   (\cZ \times_{\wt\cT} \Hilb) \times Y$ be the universal subscheme. So
we have a diagram
$$
\xymatrix{\cW    \ar[dr] \ar@{^(->}[r] &
(\cZ \times_{\wt\cT} \Hilb) \times Y \ar[d]\ar[r] & \cZ \times Y \ar[d]\\
&  \Hilb \ar[r] & {\wt\cT}}
$$
Now there exists an open subvariety $\wt H \subset {\rm Hilb}$ that parameterizes the graphs of simple branched coverings of the given type. This subspace (together with the universal object) represents the functor
$$
\wt {\underline H} : ((\text{analytic spaces}/\wt\cT ))\to ((\text{sets}))
$$
that assigns to any space $S  \to  \wt\cT$ the set of all (flat, projective families of) simple coverings
$$
\cZ\times_{\wt\cT} S  \to  Y   .
$$
We denote the universal simple covering by $(\beta,f):\cZ \times_T \wt H \to  Y\times \wt H$.

\begin{proposition}
For any fixed Riemann surface $Y$, there exists a moduli stack and the induced coarse moduli space $\cH^{n,b}(Y)$ of simple branched coverings $X \to  Y$ of degree $n$ with $b$ branch points. The moduli space $\cH^{n,b}(Y)$ is a quotient of $\wt\cH$ by $\Gamma_p/\Gamma_p[3]$ that acts in a proper, discontinuous way. For every point of $\wt \cH$, the isotropy group is
contained in ${\mathbb Z}/2{\mathbb Z}$.
\end{proposition}

\begin{proof}
The existence of a coarse moduli space in the category of algebraic spaces is known (cf.\ \cite{hgs} and references therein). It follows from the above description that it coincides with the quotient $\wt H/(\Gamma_p/\Gamma_p[3])$. The automorphisms of a simple branched covering $X \to  Y$ fix the ramification points, and like in the classical case are determined by the restrictions to the neighborhood of a ramification point. Hence the automorphism group of a fiber is
contained in ${\mathbb Z}/2{\mathbb Z}$.
\end{proof}

Just as in the situation of the classical Hurwitz space, given a finite subset of $Y$ of cardinality $b$, there are, up to isomorphism, only finitely many simple branched coverings of $Y$ of given degree that are branched exactly over the given subset of $Y$. Therefore, the canonical map
$$
\cH^{n,b} \to  \text{Sym}^b(Y) \backslash \Delta, \quad
\Delta = \{(y_1,\cdots ,y_b)
\mid
y_j=y_k \text{ for some } j\neq k \}
$$
is finite.

\section{Determinant line bundle, and Quillen metric on
generalized Hurwitz spaces}\label{se:detbdl}

Let $(\cE, k)$ be any hermitian vector bundle on $\cX$. We denote the determinant line bundle of $\cE$ in the derived category by
$$
\lambda(\cE) = \det R^\bullet f_*(\cE)  .
$$
The main result of Bismut, Gillet and Soul\'e in \cite{bgs} states the existence of a Quillen metric $h^Q$ on the determinant line bundle $\lambda(\cE)$ such that the following equality holds between its Chern form and the component in degree two of a fiber integral:
\begin{equation}\label{eq:bgs}
c_1(\lambda(\cE),h^Q)  =   \left[\int_{\cX/S}\textit{td}
(\cX/S,\omega_{\cX/S})\textit{ch}(\cE,h)\right]_2   .
\end{equation}
Here $\textit{ch}$ and $\textit{td}$ stand for the Chern character form and the Todd form respectively.

The formula in \eqref{eq:bgs} is applied to a virtual bundle of degree zero (cf.\ \cite[Remark 10.1]{f-s:extremal} for the notion of virtual bundles). It follows immediately from the definition of a determinant line bundle that the determinant line bundle of a virtual vector bundle is well-defined (as a
holomorphic line bundle).

We consider a local, universal family of simple branched coverings
\begin{equation}\label{ma1}
(\beta,f) :  \cX  \to Y\times S  .
\end{equation}
There is a holomorphic hermitian line bundle $(\cL,\wt h)$ on $Y$, whose (real) curvature form is equal to $\omega_Y$, meaning
\begin{equation}\label{ma2}
2\pi \cdot c_1(\cL, \wt h) =  \omega_Y   .
\end{equation}
Namely, if $Y =  {\mathbb P}_1$ we consider the metric of constant (Ricci) curvature equal to $+1$ and set $\cL = \cK^{-1}_{{\mathbb P}_1}$ equipped with the induced metric; for $p(Y) > 1$ we have the hyperbolic metric of constant curvature $-1$ and the induced metric on the canonical bundle $\cL = {\cK}_Y$. In case of an elliptic curve $Y$, we have the flat metric of total volume one, which is the curvature of a
holomorphic hermitian line bundle $\cL$ of degree one.

As in the setup of Section~\ref{ss:pairing}, we have the relative canonical line bundle $\cK_{\cX/S}$ with the distinguished metric $g^{-1}$ so that
$$
\omega_\cX  =  2\pi\cdot c_1(\cK_{\cX/S}, g^{-1})  .
$$
Let
$$
\cE  = \left(\cK_{\cX/S} -\cO_X\right)\otimes \left( \beta^* \cL -
\cO_\cX\right)
$$
be a virtual vector bundle of rank zero. The term of degree zero in the Chern character $ch(\cK_{\cX/S} -\cO_X)$ is equal to the (virtual) rank, which is zero; similarly, the term of degree zero in the Chern character $ch(\beta^* \cL - \cO_\cX)$ is zero. The hermitian structure on the line bundles $\cK_{\cX/S}$, $\cO_X$ and $\cL$, provide $\cE$ with a natural Chern character form. Now the first non-vanishing term of the Chern character form of $\cE$ is
$$
\frac{1}{4\pi^2} \cdot \omega_{\cX} \we \beta^*\omega_Y  ,
$$
where $\beta$ is the map in \eqref{ma1}. Hence, the only contribution of the Todd character form in \eqref{eq:bgs} is the constant $1$, resulting in the following equality
\begin{equation}\label{eq:fib0}
c_1(\lambda(\cE),h^Q) =  \frac{1}{4\pi^2} \int_{\cX/S}\omega_{\cX} \we
\beta^*\omega_Y = \frac{1}{4\pi^2} \omega^{WP}  .
\end{equation}
Both the determinant line bundle and the Quillen metric are compatible with base change. These objects descend to the generalized Hurwitz spaces in the orbifold sense. Altogether we have:

\begin{theorem}\label{th:linebundle}
There is a natural holomorphic hermitian line bundle on the generalized Hurwitz space $\cH^{n,b}(Y)$, whose Chern form coincides with the generalized \wp form up to a numerical factor.
\end{theorem}

An alternative approach, which will be considered in Section~\ref{se:delpa}, is provided by the Deligne pairing.

\section{Embedding of generalized Hurwitz spaces}

We know from \cite{artin} (cf.\ also \cite{hgs}) that a generalized Hurwitz space $\cH^{n,b}(Y)$ is an algebraic/Moishezon space. As such it possesses a compactification as a complex analytic space. (The notion of an algebraic/Moshezon space is being used
for both compact spaces and Zariski open subspaces.)

\begin{theorem}\label{th:embed}
Let $\cE = \left(\cK_{\cX/S} -\cO_X\right)\otimes \left( \beta^* \cL - \cO_\cX\right)$. Then a power of the determinant line bundle $\lambda(\cE)$ extends, as a holomorphic line bundle, to a compactification of $\cH^{n,b}(Y)$, and the Quillen metric extends to a compactification as a singular hermitian metric. A power of $\lambda(\cE)$ yields an embedding of
$\cH^{n,b}(Y)$ in a projective space.
\end{theorem}

\begin{proof}
The proof can be carried out analogous to the proof of the quasi-projectivity of the moduli space of canonically polarized varieties in \cite{sch}. Only the fiber integral $\int_{\cX/S} \omega^{n+1}$ (with $n=1$) in \cite[Proposition 4]{sch} has to be replaced by \eqref{eq:fib0}.

We consider the universal family of simple branched coverings $\cX  \to  Y$ over $\cH =  \cH^{n,b}(Y)$, which is given by a map $(\beta,f) :   \cX \to  Y\times \cH$. As in
Section~\ref{se:detbdl}, we have
$$
\omega_\cX =  2\pi\cdot c_1(\cK_{\cX/\cH},g^{-1})  .
$$
(We note that $\omega_\cX$ is semi-positive.) By \cite[Theorem I]{sch}, the form $\omega_\cX$ extends as a positive current to the total space of a (degenerating) family over some compactification $\ol{\cH}$. Now we consider the fiber integral \eqref{eq:fib0}. The aim is to apply \cite[Theorem II']{sch}. In order to show that the orbifold \ka form $\omega^{WP}$ extends as a current that is positive in the sense of currents, we follow the proof of \cite[Theorem 4]{sch}.
\end{proof}

\section{Deligne pairing on generalized Hurwitz spaces}\label{se:delpa}

Let $S$ and $\cX$ be integral schemes, and let $f : \cX  \to S$ be a proper flat family of smooth curves. If $L_0$ and $L_1$ are algebraic line bundles on $\cX$, then their \textit{Deligne pairing} is an algebraic line bundle
$$
\langle L_0  ,L_1 \rangle   \to   S
$$
\cite{De}. This line bundle $\langle L_0   ,L_1 \rangle$ carries a natural hermitian structure whenever both of the line bundles $L_1$ and $L_2$ are provided with hermitian metrics \cite{De}. See \cite[Theorem 1]{b-s} for the relationship between Deligne pairings and determinant line bundles.

Assume that $L_0$ and $L_1$ are equipped with hermitian structures $h_0$ and $h_1$ respectively. Let $h_D$ be the hermitian structure on $\langle L_0  ,L_1 \rangle$ corresponding to $h_1$ and $h_2$. Then \begin{equation}\label{dp}
c_1(\langle L_0,L_1 \rangle  ,h_D) =  \int_{{\cX}/S} c_1(L_0,h_0)\wedge
c_1(L_1,h_1)
\end{equation}
\cite[p. 144, (6.6.1)]{De}.

Let $({\mathcal L}  ,\widetilde{h})$ be the holomorphic hermitian line bundle over $Y$ in \eqref{ma2}.

\begin{proposition}
Let
$$
(\beta,f)  :   \cX  \to  Y\times \cH^{n,b}(Y)
$$
be the universal family over the generalized Hurwitz space. Then the Weil-Petersson form $\omega^{WP}_{\cH^{p,n}(Y)}$ on $\cH^{n,b}(Y)$ coincides with the curvature of the Deligne
pairing
$$
\langle \cK_{\cX/\cH^{p,n}(Y)}  , \beta^* {\mathcal L}\rangle
$$
equipped with the hermitian structure.
\end{proposition}

\begin{proof}
In view of the expression of the Weil-Petersson form in \eqref{eq:fib0}, the proposition follows from \eqref{dp}.
\end{proof}

\end{document}